\documentclass[12pt,reqno]{amsart}
\usepackage{amsthm}
\usepackage{amsthm}
\usepackage{amsaddr}
\usepackage[x11names]{xcolor}
\usepackage[linkcolor=Blue2,citecolor=red,colorlinks]{hyperref}
\usepackage[capitalize,nameinlink]{cleveref}
\usepackage{amsthm}
\usepackage{mathrsfs}

\usepackage{amsmath}
\author [ Bhowmik]{  Mainak Bhowmik}
\address{Department of Mathematics, 	Indian Institute of Science, 
	Bangalore 560012, India\\
	mainakb@iisc.ac.in}
\author [Kumar]{Poornendu Kumar}
\address{Department of Mathematics, University of Manitoba, Winnipeg R3T 2N2, Canada \\
poornendu.kumar@umanitoba.ca}
\usepackage{color, bm, amscd, tikz-cd}
\newcommand{\cA}{{\mathcal A}}
\newcommand{\cB}{{\mathcal B}}
\newcommand{\cC}{{\mathcal C}}

\newcommand{\cF}{{\mathcal F}}

\newcommand{\cH}{{\mathcal H}}

\newcommand{\cN}{{\mathcal N}}

\newcommand{\cS}{{\mathcal S}}

\newcommand{\bT}{{\mathbb{T}}}

\newcommand{\bD}{{\mathbb D}}
\newcommand{\bG}{{\mathbb G}}
\newcommand{\bC}{{\mathbb C}}

\newtheorem{thm}{Theorem}[section]

\newtheorem{lemma}[thm]{Lemma}

\newtheorem{proposition}[thm]{Proposition}

\theoremstyle{plain}
\newtheorem{definition}[thm]{Definition}
\newtheorem{remark}[thm]{Remark}

\theoremstyle{definition}

\newcounter{tmp}
\numberwithin{equation}{section}

\begin{document}
	\title{Function theory on  quotient domains related to the polydisc}

	\thanks{2020 {\em MSC}: 32A17, 46J20, 32E30, 47A56, 58D19.\\
		{\em Keywords and phrases}: Pseudo-reflection groups, Rational inner functions, \v Silov boundaries, Quotient domains,  test functions, Carath\'eodory approximation.
	}

	\maketitle
	\begin{abstract}

		Inner functions are the backbone of holomorphic function theory. This paper studies the inner functions on quotient domains of the open unit polydisc, $\bD^d$, arising from the group action of finite pseudo-reflection groups. Such quotient domains are known to be biholomorphic to the proper image $\theta(\bD^d)$ of $\bD^d$ under certain polynomial maps $\theta: \bD^d \to \theta(\bD^d)$. The main contributions of this paper are as follows:
		
		\begin{enumerate}
			\item We show that the closed algebra generated by inner functions on $\theta(\bD^d)$ forms a proper subalgebra of $H^\infty(\theta(\bD^d))$, the algebra of bounded holomorphic functions on $\theta(\bD^d)$.

			\item The set of all rational inner functions on $\theta(\bD^d)$ is shown to be dense in the norm-unit ball of $H^\infty(\theta(\bD^d))$ with respect to the uniform compact-open topology, thereby proving the Carath\'eodory approximation result.
			
			\item As an application of the Carath\'eodory approximation theorem, we approximate holomorphic functions on $\theta(\bD^d)$ that are continuous in  the closure of ${\theta(\bD^d)}$ by convex combinations of rational inner functions in the $L^2 $-norm, thereby obtaining a version of the Fisher's theorem.
			
			\item Given the two approximation results above, establishing a structure for rational inner functions is essential. We have identified the structure of rational inner functions on $\theta(\mathbb{D}^d)$.
			
			\item The Carath\'eodory approximation for operator-valued functions is also discussed.
		\end{enumerate}

	\end{abstract}

	\section{Introduction}
	
	Much like the function theory on the unit polydisc and the Euclidean unit ball captivated complex analysts \cite{BPS, BPS2, Knese, Knese2, Rudin, Rudin_EB}, the function theory on a class of domains arising from group actions on various bounded symmetric domains has become a prominent area of research \cite{SYM_GEO, Tirtha-Hari-JFA, BK2, SR, Cos, GSR,  KZ-Geo, MRZ}. The present article studies the bounded holomorphic functions on quotient domains related to the unit polydisc, $\bD^d$.
	
	Fix a positive integer $d>1$. A {\em pseudo-reflection} is a linear map $\sigma: \bC^d \rightarrow \bC^d$ such that $\operatorname{rank}(I_d -\sigma)=1$ and $\sigma^n = I_d$ for some $n \in\mathbb{N}$. A group $G$ generated by pseudo-reflections is called a {\em pseudo-reflection group}. A domain $\Omega$ is said to be $G$-invariant if, under the action $\sigma \cdot \boldsymbol{z}= \sigma^{-1}(\boldsymbol{z})$ of $G$, the domain $\Omega$ remains invariant. This action induces a natural action $$(\sigma \cdot f)(\boldsymbol{z})= f(\sigma^{-1}\cdot \boldsymbol{z})=f(\sigma(\boldsymbol{z}))$$ of $G$ on complex-valued functions on $\Omega$. A function $f$ is said to be $G$-invariant if $\sigma \cdot f = f$ for each $\sigma \in G$. Examples of pseudo-reflection groups include any finite cyclic group, the dihedral groups and the permutation groups. We denote the permutation group by $S_d$, acting on $\bC^d$ by permuting the co-ordinates, that keeps $\bD^d$ invariant. The symmetric polynomials are some examples of $S_d$-invariant functions in this case.
	
	The renowned works of Chevalley \cite{Chevalley}, Shephard and Todd \cite{Shephard-Todd} establish that for a finite linear group $G$ and a $G$-invariant domain $\Omega$, the quotient topological space $\Omega/G$ is the image $\theta(\Omega)$ under a polynomial map (called as {\em basic polynomial map associated to $G$}) $$\theta(\boldsymbol{z})=(\theta_1(\boldsymbol{z}), \dots, \theta_d(\boldsymbol{z}))$$ for $\boldsymbol{z}\in \Omega$ if and only if $G$ is a pseudo-reflection group. Although the map $\theta$ is not unique, the degrees of $\theta_j$'s are unique for $G$ up to order. Also, the $\theta_j$'s are homogeneous algebraically independent polynomials. 
	
	In this article, we shall consider $\Omega = \bD^d$ and the pseudo-reflection group $G$ to be one of the imprimitive pseudo-reflection groups $G(m,t,d)$, which appeared in the classification of finite pseudo-reflection groups, defined as follows \cite{Rudin-IUMJ}. Let $m$ and $t$ be positive integers such that $t$ divides $m$. Let $\alpha = e^{\frac{2 \pi i}{m}}$ and $G(m,t,d)$ consist of the maps 
	$$
	\boldsymbol{z}=(z_1,\dots, z_d) \mapsto \left( \alpha^{\nu_1} z_{\sigma(1)}, \dots, \alpha^{\nu_d} z_{\sigma(d)}\right)
	$$
	where $\sigma \in S_d$ and $\nu_1,\dots, \nu_d$ are integers whose sum is divisible by $t$. The group $G(m,t,d) $ is a pseudo-reflection group of order $\frac{m^d d!}{t}$. Note that the choice $(m, t)=(1,1)$ gives $G(1,1,d) = S_d$. The polynomials
	\begin{align}
		&\theta_j(\boldsymbol{z}) = E_j(z_1^m, \dots, z_d^m) \text{ for } j=1,\dots, d-1, \label{Eq:theta-j}\\
		&\theta_d(\boldsymbol{z}) = [E_d(z_1^m, \dots, z_d^m)]^{\frac{1}{t}} = (z_1 \dots z_d)^{\frac{m}{t}} \label{Eq:theta-d},
	\end{align}
	where $E_1, \dots, E_d$ are the elementary symmetric polynomials, form a basic polynomial map $\theta$ associated to $G$. Throughout this article, we shall work with this fixed polynomial map. {\em For fixed $m,d$ and $t$, we shall write $G$ to denote the group $G(m,t,d)$}. From now on, we shall work on $\theta(\bD^d)$, which is biholomorphic to the quotient domain $\bD^d /G(m,t,d)$.  In particular, the quotient domain $\bD^d/S_d$ is known as the {\em symmetrized polydisc} and it is a non-convex but polynomially convex domain, see \cite{Edigarian-Zwonek}. It is naturally associated with the famous spectral interpolation problem, see \cite{Cos2}.  The quotient domains $\Omega /G(m,t,d)$ have been studied in \cite{BBS, BK, BDGR-adv, Chen-JFA, GN-BDSM, GSR, Wick-CJM,Pal-Shalit, Rudin-IUMJ} when $\Omega$ is either the open Euclidean unit ball or the open unit polydisc. 
	
	The basic polynomial map $\theta$ is proper on $\bD^d$ and it extends as a proper map of the same multiplicity from a neighbourhood of $\overline{\bD^d}$ to a neighbourhood of $\theta(\bD^d)$. Also, the \v Silov boundary $\bT^d$ of $\bD^d$ with respect to the polydisc algebra $\cA(\bD^d)$ is the same as  $\theta^{-1}(\partial \theta(\bD^d))$ where $\partial \theta(\bD^d)$ is the \v Silov boundary of $\theta(\bD^d)$ with respect to the uniform algebra $\cA(\theta(\bD^d))$ of continuous functions on $\overline{\theta(\bD^d)}$ which are holomorphic in $\theta(\bD^d)$; see \cite{Kos-Zow}. Therefore, $\partial \theta(\bD^d)= \theta(\bT^d)$. We shall use $\boldsymbol{p}=(p_1, \dots, p_d)$ as co-ordinates in $\theta(\bD^d)$.
	
	Let $H^\infty(\theta(\bD^d))$ be the Banach algebra of all bounded holomorphic functions on $\theta(\bD^d)$ with sup-norm. If $g \in H^\infty(\mathbb{D}^d) $, the Banach algebra bounded analytic functions on $\bD^d$, then the radial limit
	$$
	g^*(\zeta_1,\cdots, \zeta_d):= \lim_{r\to 1^{-}} g(r\zeta_1,\cdots, r\zeta_d)
	$$
	exists almost everywhere in $\mathbb{T}^d$ with respect the normalized Haar measure $\nu$ on $\mathbb{T}^d$; see \cite[Chapter 3]{Rudin}. For every $f\in H^\infty(\theta(\bD^d))$, $f\circ \theta$ is in $H^\infty(\bD^d)$ having a radial limit. Therefore, the boundary values of $f$ exist and the boundary value function $f^*$ is given by 
	$$
	f^*(q_1, \dots, q_d)= \lim_{r\to 1^{-}} f(r^m q_1, r^{2m}q_2, \dots, r^{(d-1)m}q_{d-1}, r^{dm/t}q_d) 
	$$ 
	for almost every (with respect to $\mu$) $\boldsymbol{q}=(q_1, \dots, q_d) \in \partial \theta(\bD^d)$ where $\mu$ is the push-forward of the measure $\nu$ on $\mathbb{T}^d$ under the map $\theta$. Here, we are using the explicit expressions for $\theta_j$'s in terms of the elementary symmetric polynomials. Also, it is easy to check that $f \mapsto f^*$ is an isometric embedding of $H^\infty(\theta(\bD^d))$ into $L^\infty(\partial \theta(\bD^d), \mu)$; for a detailed discussion see \cite[Lemma 2.2]{BBS}.
	
	A function $f$ in $H^\infty(\theta(\bD^d))$ is said to be {\em inner} if $f^*$ is unimodular almost everywhere with respect to $\mu$ on $\partial \theta(\bD^d)$. A {\em rational inner function} $f$ on  $\theta(\bD^d)$ is a rational function  with poles off  $\theta(\bD^d)$ which is also inner. An important family of examples of rational inner functions on the symmetrized bidisc, $\bG_2$ which is biholomorphic to $\bD^2/G(1,1,2)$, is $$\varphi_{\beta}(p_1,p_2)= \frac{2\beta p_2-p_1}{2-\beta p_1}$$ for each $\beta$ in the unit circle $\mathbb{T}$. See \cite{SYM_GEO, SYM_Real, Tirtha-Hari-JFA} for its importance. Rational inner functions appear as solutions to the Pick-Nevanlinna interpolation problem, see \cite{DKS, DKS2, KZ2}.

	We first ask whether the algebra generated by inner functions is large enough to give us any advantage. In this direction, we have the following result.
	\begingroup
	\setcounter{tmp}{\value{thm}}
	\setcounter{thm}{0}
	\renewcommand\thethm{\Alph{thm}}
	\begin{thm}\label{Theorem B}
		The closed algebra generated by the inner functions in $H^\infty( \theta(\bD^d))$ is a proper subalgebra of $H^\infty(\theta(\bD^d))$.
	\end{thm}
	\endgroup
	
	The proof requires us to find a new characterization of inner functions on $\theta(\bD^d)$. We also show that unlike the classical case of $\mathbb{D}$, the \v Silov boundary of $H^\infty(\theta(\bD^d))$ is a proper subset of the set of all restrictions of complex homomorphisms on $L^\infty(\partial \theta(\bD^d), \mu)$. These are the contents of Section \ref{Alg-results}.
	
	It may be a surprise that in spite of Theorem \ref{Theorem B}, we have approximation results. Rational inner functions on the unit disc have been greatly studied for their usefulness. A classical theorem of Carath\'eodory says that any holomorphic self map of $\mathbb{D}$ can be approximated uniformly on compact subsets of $\mathbb{D}$ by rational inner functions. For modern treatment of this theorem, see \cite{B_J_K, BBK}. We prove such an approximation result in quotient domains.  
	
	\begingroup
	\setcounter{tmp}{\value{thm}}
	\setcounter{thm}{1}
	\renewcommand\thethm{\Alph{thm}}
	\begin{thm} \label{Theorem C}
		Any holomorphic function $f:\theta(\bD^d)\rightarrow{\overline{\mathbb{D}}}$ can be approximated (uniformly on compact subsets) by rational inner functions in $\cA(\theta(\bD^d))$.
	\end{thm}
	\endgroup
	
	This in turn enables us to prove a theorem on approximation by convex combinations of rational inner functions in the $L^2$ sense. These are the contents of Section \ref{Apr-Inner}.
	
	\begingroup
	\setcounter{tmp}{\value{thm}}
	\setcounter{thm}{2}
	\renewcommand\thethm{\Alph{thm}}
	\begin{thm} \label{Theorem D}
		Any function in the norm unit ball of $H^\infty \left(\theta(\bD^d) \right)$ can be approximated by convex combinations  of rational inner functions in $\cA( \theta(\bD^d))$ with respect to the $L^2$-norm on $ \partial \theta(\bD^d) $ equipped with the measure $\mu$.
	\end{thm}
	\endgroup

	In Theorem \ref{Theorem C} and Theorem \ref{Theorem D}, we approximate using rational inner functions, which prompts an exploration of their underlying structure. We establish the structure of rational inner functions on $\theta(\mathbb{D}^d)$. This is our Theorem \ref{Theorem A}.
	
	\begingroup
	\setcounter{tmp}{\value{thm}}
	\setcounter{thm}{3}
	\renewcommand\thethm{\Alph{thm}}
	\begin{thm}\label{Theorem A}
		Given a rational inner function $f$ on $\theta(\bD^d)$, there exist a non-negative integer $k$, $\tau\in\mathbb{T}$ and a polynomial $g$ with no zero in $\theta(\bD^d)$ such that
		\begin{align}\label{RIF-G}
			f(p_1, \dots, p_{d-1}, p_d)=\tau p_d^{k} \ \frac {\overline{g\left(\frac{\overline{p_{d-1}}}{\bar{p}_d^t}, \frac{\overline{p_{d-2}}}{\bar{p}_d^t}, \dots, \frac{\overline{p_1}}{\bar{p}_d^t}, \frac{1}{\bar{p}_d} \right)}}{g\left(p_1,\dots,p_{d-1}, p_d\right)}
		\end{align}
		Conversely, any rational function of the form \eqref{RIF-G} is inner. Moreover, any inner function $f\in \cA(\theta(\bD^d))$ is a rational function of the form \eqref{RIF-G} with the additional property that $g$ has no zeros in $\overline{\theta(\bD^d)}$.
	\end{thm}
	\endgroup
	This result was obtained in the case of the polydisc by Pfister \cite{Pfister} as well as Rudin and Stout \cite{Rudin-Stout}. In the case of general bounded symmetric domains, it was proven by Kor\'anyi-V\'agi \cite{Koranyi-Vagi}. We cannot apply the Korányi-Vági result to $\theta(\bD^d)$ as it is not a bounded symmetric domain. The crux of our proof lies in being able to choose the polynomial in Rudin-Stout's theorem in a certain way. This is discussed in Section \ref{RIF}.
	
	Building on the success of previous results, a natural next step is to explore their generalization to the operator-valued setting. In this context, the notion of inner functions is extended by replacing the unimodular condition with values in the space of isometries. Establishing these results in the operator-valued framework required leveraging concepts from function-theoretic operator theory on these domains. We first construct a collection of \textit{test functions} (see Section \ref{Operator-Valued} for the definition) for $\theta(\mathbb{D}^d)$ 
	when $d=2$.  Using the theory developed by Dritschel and McCullough \cite{DM}, we can obtain the {\em Pick-Nevanlinna interpolation} for a class of functions in $H^\infty(\theta(\mathbb{D}^2))$. This then enables us to provide a version of the Carathéodory approximation theorem for operator-valued functions on $\theta(\bD^2)$.
	
	\vspace{2mm}
	
	\noindent \textbf{Acknowledgement:} The authors are grateful to their research supervisor, Prof. Tirthankar Bhattacharyya, and to Dr. Haripada Sau for some discussions. The second author also thanks Prof. D. E. Marshall for reading the introduction and providing helpful comments. Finally, we sincerely appreciate the referee's comments and suggestions, particularly for drawing our attention to the paper \cite{Rudin-Stout}.

	\section {Algebra generated by inner functions} \label{Alg-results}
	A half-century ago, Marshall proved that the closed algebra generated by inner functions on the disc is $H^\infty(\mathbb{D})$ \cite{Mar}. From this it follows that the inner functions seperate the points in the maximal ideal space of $H^\infty(\mathbb{D})$. Marshall's result has also been generalized for certain uniform algebras \cite{Ber-Gar-Mar}. After that, Kon showed that these results do not hold for $H^{\infty}(\mathbb{D}^d)$ when $d>1$; see \cite{Kon, Kon2}. Naturally, we can ask these questions in our setting of quotient domains. In this section, we answer them in a negative direction for $\theta(\mathbb{D}^d)$.
	The main goal of this section is to prove Theorem \ref{Theorem B}.

	To establish this, we shall rely on a set of lemmas and propositions. Some of these lemmas can be deduced using the results from the polydisc, while others require their own distinct treatment. Furthermore, certain propositions hold independent interest and importance, even beyond their relevance to the theorem.
	
	Let $M_1$ and $M_2$ be the maximal ideal spaces of $H^\infty(\theta(\bD^d))$ and $H^\infty(\mathbb{D}^d)$ respectively. The maximal ideal spaces $M_1$ and $M_2$ are endowed with the weak$^*$ topologies inherited from the norm-unit balls of the continuous duals of the Banach spaces $H^\infty(\theta(\bD^d))$ and $H^\infty(\mathbb{D}^d)$ respectively. Also, we denote the maximal ideal spaces of $L^\infty (\partial\theta(\bD^d), \mu)$ and $L^\infty(\mathbb{T}^d, \nu)$ by $X_1$ and $X_2$ respectively. Define two maps $\tau_j : X_j \to M_j $ for $j=1,2$ in the following way:
	$$
	\tau_1(m)= m|_{H^\infty(\theta(\bD^d))} \ \text{for all}\ m\in X_1,
	$$
	$$
	\tau_2(m)= m|_{H^\infty(\mathbb{D}^d)}\ \text{for all}\ m\in X_2.
	$$
	For $f\in H^\infty(\theta(\bD^d))$, the Gelfand transform $\hat{f}$ of $f$ is a continuous function from $M_1$ to $\mathbb{C}$ given by
	$$
	\hat{f}(m)= m(f) \ \text{ for each } m\in M_1. 
	$$
	Similarly, the Gelfand transforms of functions in $L^\infty (\partial\theta(\bD^d), \mu), L^\infty (\bT^d, \nu)$ and $H^\infty(\bD^d)$ are defined on their respective maximal ideal space. See \cite{Gamelin} for more details.
	
	A closed boundary for $H^\infty(\theta(\bD^d))$ is a closed subset $\cC$ of the maximal ideal space of $H^\infty(\theta(\bD^d))$ such that $$\|f\|_\infty= \operatorname{sup}\{|\varphi(f)|: \varphi\in \cC \}$$ for all $f$ in $H^\infty(\theta(\bD^d))$. The {\em \v Silov boundary} of the uniform algebra $H^\infty(\theta(\bD^d))$ is the smallest closed boundary of $H^\infty(\theta(\bD^d))$. Let $\partial$ denote the \v Silov boundary for $H^\infty(\theta(\bD^d))$.
	For any $m\in X_2$, we can define $\tilde{m}: H^\infty(\theta(\bD^d)) \to \mathbb{C}$ such that $$\tilde{m}(\psi)= m(\psi\circ \theta).$$ Then $\tilde{m} \in X_1$. In a similar way, every member of $M_2$ gives rise to a unital complex homomorphism in $M_1$. The following proposition characterizes the inner functions on the polydisc in terms of the homomorphisms in $\tau_2(X_2)$. It is due to Kon.
	\begin{proposition}{\cite[Corollary 4]{Kon}}\label{bidisc-inner}
		A function $u$ in $H^\infty (\mathbb{D}^d)$ is inner if and only if $|\hat{u}(\Phi)|=1$ for all $\Phi \in \tau_2(X_2)$.
	\end{proposition}
	\begin{lemma}\label{Prop-Inner}
		A function $f$ in $H^\infty (\theta(\bD^d))$ is inner if and only if $|\hat{f}(\Phi)|=1$ for all $\Phi \in \tau_1(X_1)$.
	\end{lemma}
	
	\begin{proof}
		If $f$ is inner in $H^\infty(\theta(\bD^d))$, then it straightforward that $|\hat{f}(\Phi)|=1$ for all $\Phi\in \tau_1(X_1)$.
		For the converse part, assume that $|\hat{f}(\Phi)|=1$ for all $\Phi$ in $\tau_1(X_1)$. It is enough to show that $f\circ \theta$ is inner function in $\mathbb{D}^d$. Let $m \in X_2$. Then, $$|\tau_2(m)(f\circ\theta)|= |m|_{H^\infty(\mathbb{D}^d)}(f\circ \theta)| = |\tau_1(\tilde{m})(f)|=|\hat{f}(\tau_1(\tilde{m}))|=1.$$ Since $m\in X_2$ is arbitrary, Proposition \ref{bidisc-inner} implies that $f\circ \theta$ is an inner function. We note from the proof that the whole of $\tau_1(X_1)$ may not be needed for the converse part.
	\end{proof}
	
	For a non-negative Borel measure $\eta$ on $\mathbb{T}^d$ and a function $f$ in $L^1(\mathbb{T}^d, \nu)$, we shall denote the real measure $f d\nu - d\eta$ by $\eta_f$. Let $\mathcal{P}$ be the Poisson kernel of $\bD^d$. Refer to \cite{Rudin} for more details on the Poisson kernel.
	
	\begin{lemma}\label{sing}
		Let $f$ be a positive, $G$-invariant and lower semi-continuous function on $\mathbb{T}^d$ such that $f\in L^1(\mathbb{T}^d, \nu)$. Then there exists a positive measure $\eta$ on $\mathbb{T}^d$ which is singular with respect to $\nu$ such that
		$$P[\eta_f]= Re(g)$$ for some $G$-invariant holomorphic function $g$ on $\mathbb{D}^d$, where $P[\eta_f]$ is the Poisson integral of the real Borel measure $\eta_f$ on $\mathbb{T}^d$.
	\end{lemma}
	\begin{proof}
		We use Theorem 2.4.2 of \cite{Rudin} to obtain a non-negative measure $\omega$ on $\bT^d$ which is singular with respect to $\nu$ and $P[\omega_f]= Re(u)$ for some holomorphic function $u$ in $\bD^d$. For $\sigma \in G$, we consider the pull-back $\sigma^* \omega$ of $\omega$ under the homeomorphism $\sigma$ on $\bT^d$ and define a new measure $\eta$ on $\bT^d$ as follows:
		\begin{align*}
			\eta(E) := \frac{1}{|G|} \sum_{\sigma \in G} \sigma^* \omega(E) = \frac{1}{|G|} \sum_{\sigma \in G} \omega(\sigma(E))
		\end{align*}
		for each $\omega$-measurable subset $E$ of $\bT^d$. It is easy to observe that the support of the measure $\eta$ is the union of the supports of the measures $\sigma^* \omega$ for all $\sigma$ in $G$. Using a change of variables for the maps $\sigma : \bT^d \rightarrow \bT^d$ we observe that supports of the measures $\sigma^* \omega$ have $\nu$ measure zero. Hence $\eta$ is singular with respect to $\nu$. 
		
		Again $P[\omega_f]= Re(u)$ implies that for every $\sigma \in G$,
		\begin{align}\label{Eq:RP measure}
			Re \ u(\sigma(\boldsymbol{z})) &= P[f](\sigma(\boldsymbol{z})) - P[\omega](\sigma(\boldsymbol{z}))  \notag \\
			&= \int_{\bT^d} \mathcal{P}(\sigma(\boldsymbol{z}), \boldsymbol{\zeta}) f(\boldsymbol{\zeta}) d\nu(\boldsymbol{\zeta}) - \int_{\bT^d} \mathcal{P}(\sigma(\boldsymbol{z}), \boldsymbol{\zeta}) d\eta(\boldsymbol{\zeta}).                            
		\end{align}
		Suppose, $$\sigma(\boldsymbol{z})= \left( e^{\frac{2\pi i}{m}\nu_1} z_{\gamma(1)}, \dots, e^{\frac{2\pi i}{m}\nu_d} z_{\gamma(d)}\right)$$ for some $\gamma \in S_d$ and the integers $\nu_j$ are as in the definition of the group $G=G(m,t,d)$. Then by the change of variable $$\zeta_j = e^{\frac{2\pi i}{m}\nu_j} \zeta_{\gamma(j)}$$ for each $j=1,\dots, d$ and the invariance of the measure $d\nu$ under this transformation we get
		\begin{align*}
			\int_{\bT^d} \mathcal{P}(\sigma(\boldsymbol{z}), \boldsymbol{\zeta}) f(\boldsymbol{\zeta}) d\nu(\boldsymbol{\zeta}) &= \int_{\bT^d} \prod_{j=1}^d \frac{1-|z_{\gamma(j)}|^2}{|1- e^{\frac{2\pi i}{m}\nu_j} z_{\gamma(j)} \overline{\zeta_j}|^2 } f(\boldsymbol{\zeta}) d\nu(\boldsymbol{\zeta}) \\
			&= \int_{\bT^d} \prod_{j=1}^d \frac{1-|z_{\gamma(j)}|^2}{| 1- z_{\gamma(j)} \overline{\zeta_{\gamma(j)}} |^2} f(\sigma(\boldsymbol{\zeta})) d\nu(\boldsymbol{\zeta}) \\
			&= P[f](\boldsymbol{z}).
		\end{align*} 
		The penultimate equality holds as $f$ is $G$-invariant. 
		Finally, note that by our construction the measure $\eta$ is also $G$-invariant and hence the same change of variables trick as above shows that
		\begin{align*}
			\int_{\bT^d} \mathcal{P}(\sigma(\boldsymbol{z}), \boldsymbol{\zeta}) d\eta(\boldsymbol{\zeta})= P[\eta](\boldsymbol{z}).
		\end{align*} 
		Thus from \eqref{Eq:RP measure} we conclude that $$Re \ u(\sigma(\boldsymbol{z})) = Re \ u(\boldsymbol{z}) = P[\eta_f](\boldsymbol{z}).$$ This completes the proof.
	\end{proof}
	
	We shall denote the subalgebra of all $G$-invariant holomorphic functions in $H^\infty(\bD^d)$ by $H^\infty(\bD^d)^G$.
	\begin{lemma}\label{Lem3}
		
		Given a positive, bounded and $G$-invariant lower semi-continuous function $\psi$ on $\mathbb{T}^d$, there exists $f$ in $H^\infty(\bD^d)^G$ such that
		$$\psi=|f^*| \quad \nu-\text{a.e. on } \mathbb{T}^d, $$
		where $f^*$ is the radial limit of $f$ which is defined $\nu-$a.e. in $\mathbb{T}^d$.
	\end{lemma}
	\begin{proof}
		Since $\psi$ is a lower-semicontinuous function, it attains its minimum on $\mathbb{T}^d$. After adding a positive constant we can make it bigger than $1$. Hence, we shall assume that $\psi>1$. Thus $\operatorname{log}\psi$ is well defined and satisfies all the conditions of Lemma \ref{sing}. Thus we get a positive measure $\eta$ on $\mathbb{T}^d$ which is singular with respect to $\nu$, and a $G$-invariant holomorphic function $g$ on $\mathbb{D}^d$ such that $$P[\eta_{\operatorname{log}(\psi)}]= Re(g).$$
		Set $f=\operatorname{exp}(g)$. By the definition of $f$, it is $G$-invariant and holomorphic on $\mathbb{D}^d$. Now we shall show that $f$ is bounded. To that end, note that
		$$\operatorname{log}|f|= \operatorname{Re}(g)=P[\eta_{\operatorname{log}(\psi)}]= P[{\operatorname{log}(\psi)}]-P[d\eta].$$
		As $\eta$ is a non-negative measure, we get
		$$P[\eta_{\operatorname{log}(\psi)}]\leq P[\operatorname{log}(\psi)].$$
		This imples that $\operatorname{log}|f|\leq P[\operatorname{log}(\psi)]$. Since $P[\operatorname{log}(\psi)]$ is bounded, $\operatorname{log}|f|$ is bounded. Thus, $f$ is bounded and hence $f\in H^{\infty}(\mathbb{D}^d)^G$. Applying Theorem 2.3.1 of \cite{Rudin} to the function $f$ and the measure $\eta$, we get that $\psi=|f^*|$.
	\end{proof}
	Consider the maximal ideal spaces $M_{0}$ and $X_0$ of the algebras $H^\infty(\mathbb{D})$ and $L^\infty(\mathbb{T})$, respectively. It is well known that the map $\tau_0: X_0 \rightarrow M_0$, defined as $\tau_0(m)= m|_{H^\infty(\mathbb{D})}$ for all $m\in X_0$, is a homeomorphism from $X_0$ onto the \v Silov boundary of $H^\infty(\mathbb{D})$, see \cite{Hoff}. In the subsequent theorem, we shall establish that, unlike the classical case of $\bD$, the \v Silov boundary is a proper subset of the set comprising of all restrictions of complex homomorphisms on $L^\infty(\partial \theta(\bD^d), \mu)$. The proof is motivated by \cite{Range}.

	\begin{thm}\label{prop1}
		The map $\tau_1: X_1\to M_1$ is continuous and $\partial $ is a proper subset of $\tau_1(X_1)$.
	\end{thm}
	\begin{proof}
		Take a countable dense subset $\{\boldsymbol{p}^{(n)}=(p^{(n)}_1, \dots, p^{(n)}_d): n\in \mathbb{N}\}$ of $\partial \theta(\bD^d)$. For each $n$, define 
		$$A_n= \{(e^{i m s}p^{(n)}_1, \dots, e^{i (d-1) m s}p^{(n)}_{d-1}, e^{i d m s/t}p^{(n)}_d ): s \in [0,2\pi]\}.$$ Each $A_n$ is a compact set. We consider tubular type open neighbourhoods ${A_{j,n}}$ of the set $A_n$ such that $\mu (A_{j,n})< \frac{1}{j.2^{n+1}}$ for every $j\in \mathbb{N}$. Moreover, we choose these neighbourhoods in such a way that, if $(p_1, \dots, p_d)\in A_{j,n}$, then $$(e^{i m s}p_1, \dots, e^{i (d-1) m s}p_{d-1}, e^{i d m s/t}p_d )\in A_{j,n}$$ for all $s \in [0,2\pi]$. Take, $E_j =\cup_{n\in \mathbb{N}} A_{j, n} $. Clearly, $\mu(E_j) \leq \frac{1}{2j}$ and $E_j$ is an open dense subset of $\partial \theta(\bD^d)$.
		
		Consider the set $Z_k= \{\varphi \in X_1: \varphi(\chi_{E_k})=1 \}$ where $\chi_{E_k}$ is the indicator function of the set $E_k$. Then the Gelfand transform of $\chi_{E_k}$ satisfies the condition $\hat{\chi}_{E_k} = \chi_{Z_k}$.
		
		For $f\in L^\infty(\partial \theta(\bD^d), \mu)$, the Gelfand transform $\hat{f}$ is in $C(X_1)$. Define a continuous linear functional on $C(X_1)$ given by, $\hat{f} \mapsto \int_{\partial \theta(\bD^d)}f d\mu$. Then by the Riesz representation theorem, there exists a unique regular Borel measure $\hat{\mu}$ such that $$\int_{\partial \theta(\bD^d)} f d\mu = \int_{X_1} \hat{f} d\hat{\mu}.$$
		Therefore
		\begin{align}\label{Eq:mu-hat}
			\hat{\mu}(Z_k) &= \int_{X_1} \chi_{Z_k} d\hat{\mu}
			= \int_{X_1} \hat{\chi}_{E_k}d\hat{\mu}
			= \int_{\partial \theta(\bD^d)} \chi_{E_k} d\mu
			= \mu(E_k) \leq \frac{1}{2k}.
		\end{align}
		
		Let $f \in H^\infty ( \theta(\bD^d))$. Define, $L_f : \partial \theta(\bD^d) \to \mathbb{R}$ by 
		\begin{align*}
			L_f(\boldsymbol{p}) &= \operatorname{ess \ sup}_{s \in [0,2\pi]} |f^*(e^{i m s}p_1, \dots, e^{i (d-1) m s}p_{d-1}, e^{i d m s/t}p_d )| \\
			&= \sup_{0<r<1} \left\lbrace \sup_s |f(r^m e^{i m s}p_1, \dots, r^{m(d-1)} e^{i (d-1) m s}p_{d-1}, r^{m/t}e^{i d m s/t}p_d )| \right\rbrace.
		\end{align*}
		
		Then $L_f$ can be shown to be lower semi-continuous on $\partial \theta(\bD^d)$ using a similar argument used in Theorem 3.5.2 in \cite{Rudin}. Now suppose that
		$$\sup_{\tau_1(Z_k)} |\hat{f}|= \operatorname{ess \ sup}_{E_k} |f^*| \leq 1.$$ For $\mu$-a.e, $\boldsymbol{p}\in E_k$ , $L_f(\boldsymbol{p})  \leq 1. $ Since $L_f$ is lower semi continuous and $E_k$ is open dense in $\partial \theta(\bD^d)$, for all $\boldsymbol{p}\in \partial \theta(\bD^d)$, $L_f(\boldsymbol{p}) \leq 1 $ and hence $\|f\|_\infty \leq 1.$ Therefore, $\tau_1(Z_k)$ is a closed boundary for $H^\infty(\theta(\bD^d))$.
		We can construct $E_j$'s such that $E_j \supset E_{j+1}$ and so, $Z_j \supset Z_{j+1}$. $Z_j$'s being non empty compact subsets, by the finite intersection property we have, $$ E= \cap_{j=1}^\infty Z_j \neq \emptyset.$$ Also, $$\tau_1(E)= \cap_{j=1}^\infty \tau_1(Z_j).$$ We know that $\tau_1(Z_j)$'s are closed boundaries for $H^\infty (\theta(\bD^d))$ and hence $\tau_1(E)$ is so. Note that, \eqref{Eq:mu-hat} implies $$\hat{\mu}(E)=\lim_{j\to \infty} \hat{\mu}(Z_j)=0.$$ Also, the closed support of $\hat{\mu}$ is $X_1$ (This can be easily seen from Lemma 9.1 in \cite{Gamelin}). So, we have shown that there exists a nowhere dense subset $E$ of $X_1$ with $\hat{\mu}(E)=0$ such that $\tau_1(E)$ is a closed boundary for $H^\infty(\theta(\bD^d))$.
		
		Consider a subset $A\subseteq \partial \theta(\bD^d) \setminus E_1$ with $\mu(A)>0.$ Then $\theta^{-1}(A)$ is $G$-invariant subset of positive $\nu$ measure in $\mathbb{T}^d$. Now we take the lower semi continuous function $\psi$ given by
		\begin{equation*}
			\psi=
			\begin{cases}
				1, & \text{on}\ \theta^{-1}(A)  \\
				2, & \text{on}\ \theta^{-1}(E_1)\\
				0, & \text{else}.
			\end{cases}
		\end{equation*}
		By Lemma \ref{Lem3}, there exists a $G$-invariant bounded holomorphic function $g$ such that \begin{equation*}
			|g^*|=
			\begin{cases}
				1, & \text{on}\ \theta^{-1}(A) \\
				2, & \text{on}\ \theta^{-1}(E_1).
			\end{cases}
		\end{equation*}
		This means that there exists $f\in H^\infty(\theta(\bD^d))$ such that
		\begin{equation*}
			|f^*|=
			\begin{cases}
				1, & \text{on}\ A \\
				2, & \text{on}\ E_1.
			\end{cases}
		\end{equation*}
		
		Therefore, $|\hat{f}|=2$ on the \v Silov boundary of $H^\infty(\theta(\bD^d))$. But $|\hat{f}|$ is not identically $2$ on $\tau_1(X_1)$. Indeed, consider
		$$
		Z_A=\{\varphi \in X_1: \hat{\chi}_A(\varphi)=1\} \,\text{and}\
		Z_{E_1}=\{\varphi \in X_1: \hat{\chi}_{E_1}(\varphi)=1 \}.
		$$
		We can find two unimodular measurable functions $g_1$ and $h$ on $\partial \theta(\bD^d)$ such that $$f^*= 2g_1 \chi_{E_1} + h \chi_A.$$ For $\varphi \in Z_{E_1}$, we have $\varphi|_{H^\infty(\theta(\bD^d))} \in \tau_1(Z_{E_1})$ and
		\begin{align*}
			|\hat{f}(\varphi|_{H^\infty(\theta(\bD^d))})| &= |\varphi(f^*)|
			= 2|\varphi(g_1) \varphi(\chi_{E_1}) + \varphi(h) \varphi(\chi_A)|
			=2
		\end{align*} as $|\varphi(h)|= |\varphi(g_1)| =1$ and $\varphi(\chi_A)=0$.
		So, $|\hat{f}|=2$ on $\tau_1(Z_{E_1})$. We know that $\tau_1(Z_{E_1})$ is a closed boundary for $H^\infty(\theta(\bD^d))$. Hence $|\hat{f}|=2$ on the \v Silov boundary of $H^\infty(\theta(\bD^d))$. Similarly, $|\hat{f}|=1$ on $\tau_1(Z_A)$. Since $\mu(A)>0$, $\tau_1(Z_A) \neq \emptyset$ and so $\tau_1(Z_A) \cap \partial  = \emptyset$. Therefore, $\partial $ is a proper subset of $\tau_1(X_1)$.
	\end{proof}
	It is worth noting that the proof of the theorem above not only confirms this result but also provides additional insights. Specifically, it demonstrates the existence of a nowhere dense subset in the maximal ideal space of $L^\infty(\partial \theta(\bD^d), \mu)$, which serves as a closed boundary for $H^\infty(\theta(\bD^d))$. The theorem mentioned above, along with the following proposition, will play a crucial role in proving the main theorem.
	\begin{proposition}\label{prop2}
		The Gelfand transforms of the inner functions in $H^\infty(\theta(\bD^d))$ cannot separate points of the maximal ideal space $M_1$.
	\end{proposition}
	\begin{proof}
		The proof of this proposition is standard once we have Theorem \ref{prop1}. Indeed, If possible, suppose that the Gelfand transforms of the inner functions in $H^\infty(\theta(\bD^d))$ separate points of $\tau_1(X_1)$. Take a proper compact subset $K\subset \tau_1(X_1)$ and let $\Phi'\in \tau_1(X_1)\setminus K$. Then for each $\Phi_{\alpha}	\in K$, by our assumption, there exists an inner function $I_{\alpha}$ such that
		$$\hat{I}_{\alpha}(\Phi')\neq \hat{I}_{\alpha}(\Phi_{\alpha}).$$
		By Lemma \ref{Prop-Inner}, both $\hat{I_{\alpha}}(\Phi')$ and $\hat{I}_{\alpha}(\Phi_{\alpha})$ are unimodular. So, without loss of generality we can assume that $$\hat{I}_{\alpha}(\Phi')=1 \quad \text{ and } \quad  Re\left(\hat{I}_{\alpha}(\Phi_{\alpha})\right)<1.$$
		By continuity, there exists a neighbourhood $\cN_\alpha$ of $\Phi_{\alpha}$ such that $Re\left(\hat{I}_{\alpha}(\Phi)\right)<1$ for all $\Phi$ in  $\cN_{\alpha}$. Since $K$ is a compact set, there exists $r$ such that $K\subset \cup_{j=1}^r \cN_{\alpha_j}$. Hence, for each $j=1,2,\dots, r$, we have  $\hat{I}_{\alpha}(\Phi')=1.$ Also,
		$$\operatorname{inf}_{1\leq j\leq r} Re\left(\hat{I}_{\alpha_j}(\Phi)\right)<1 \quad \text{ for all } \Phi\in K.$$
		Consider the holomorphic function $g=1+I_{\alpha_1}+\dots +I_{\alpha_r}$ in $H^\infty(\theta(\bD^d))$. Note that,
		$$\hat{g}(\Phi')= \hat{1}(\Phi')+\hat{I}_{\alpha_1}(\Phi')+\dots +\hat{I}_{\alpha_r}(\Phi')=r+1= \|\hat{g}\|.$$
		Also for $\Phi\in K$, we have
		$$|\hat{g}(\Phi)|=|\hat{1}(\Phi)+\hat{I}_{\alpha_1}(\Phi)+\dots +\hat{I}_{\alpha_r}(\Phi)|\leq r+1.$$
		Now, if $|\hat{g}(\Phi)|=r+1$, then $\hat{I}_{\alpha_1}(\Phi)+\dots +\hat{I}_{\alpha_r}(\Phi)=r.$ Thus, $\sum_{j=1}^r Re\left(\hat{I}_{\alpha_j}(\Phi)\right) =r$, which is a contradiction as $$\operatorname{inf}_{1\leq j\leq r}Re\left(\hat{I}_{\alpha_j}(\Phi)\right)<1.$$ Therefore, for each $\Phi\in K$, we have $|\hat{g}(\Phi)|<r+1$ and $|\hat{g}(\Psi)|=r+1$. This implies that $K$ cannot be a closed boundary of $H^\infty(\theta(\bD^d))$. Since $K$ is an arbitrary proper compact subset of $\tau_1(X_1)$, we must have that $\tau_1(X_1)$ is the \v Silov boundary for $H^\infty(\theta(\bD^d))$. This is not possible because of Theorem \ref{prop1}. Hence, the Gelfand transforms of inner functions cannot separate points of $\tau_1 (X_1)$ and hence the points of $M_1$.
	\end{proof}
	
	Now, we are ready to prove the main theorem of this section.
	
	\begin{proof}[\textbf{Proof of Theorem \ref{Theorem B}:}] Let $\mathcal{A}$ be the closed algebra generated by inner functions in $H^\infty(\theta(\bD^d))$. Suppose that $\mathcal{A}= H^\infty(\theta(\bD^d))$. Let $m_1$ and $m_2$ be two distinct elements in $M_1$. Then there exists $f \in H^\infty(\theta(\bD^d))$ such that $m_1(f)\neq m_2(f)$. If for every inner function $I$, $m_1(I)=m_2(I)$ then $m_1 = m_2$ on $\mathcal{A}$ as the inner functions generate it. This is a contradiction. Hence the Gelfand transforms of the inner functions separate points of $M_1$. This contradicts Proposition \ref{prop2}. Hence $\mathcal{A}$ is a proper subalgebra of $H^\infty(\theta(\bD^d))$.
	\end{proof}

	\section{Approximation by Inner Funtions } \label{Apr-Inner}
	This section aims to discuss some approximation results by rational inner functions. We first show that the set of rational inner functions is dense in the norm-unit ball of $H^\infty(\theta(\bD^d))$ with respect to the compact-open topology. In a recent work \cite[Section 7]{BBS}, the density of rational inner functions in $\theta(\bD^d)$ leads to an interesting result about projectivity of Hilbert modules over $\cA(\theta(\bD^d))$ in a certain category.
	
	We use the following notations: $\boldsymbol{z} = (z_1, z_2,\dots, z_d)$, $\boldsymbol{z}^{\boldsymbol{n}} = z_1^{n_1} z_2^{n_2}\dots  z_d^{n_d}$ for $\boldsymbol{n}=(n_1, \dots, n_d)$ and $\frac{1}{\overline{\boldsymbol{z}}}=(\frac{1}{\overline{z_1}}, \frac{1}{\overline{z_2}},\dots, \frac{1}{\overline{z_d}})$. The idea of the proof stems from Satz 4 in \cite{Pfister} as well as Theorem 5.5.1 in \cite{Rudin}.
	
	\begin{proof}[\textbf{Proof of Theorem \ref{Theorem C}:}]
		Choose $\epsilon >0$. Let us fix a compact subset $S$ of $\theta(\bD^d)$. Since the map $\theta$ is a proper map, $K=\theta^{-1}(S)$ is a compact subset of $\bD^d$. The function $f \circ \theta$ is $G$-invariant and holomorphic in $\bD^d$. Also, for $\boldsymbol{z} \in K$,  $$\theta(\sigma(\boldsymbol{z}))=\theta(\boldsymbol{z}) \in K$$ i.e., $\sigma(\boldsymbol{z}) \in K$ for every $\sigma \in G$. We choose a polynomial $P$  such that $|P|<1$ on $\overline{\theta(\bD^d)}$ and 
		$$
		|f \circ \theta(\boldsymbol{z}) - P(\boldsymbol{z})|< \epsilon \ \text{ for } \boldsymbol{z} \in K.
		$$ 
		Consider $$P_G(\boldsymbol{z})= \frac{1}{|G|} \sum_{\sigma \in G} P(\sigma(\boldsymbol{z})).$$ Clearly, $|P_G|<1$ and  
		\begin{align}\label{Eq:estimate of polynomial}
			\left|f\circ \theta(\boldsymbol{z}) - P_G(\boldsymbol{z})\right|= |\frac{1}{|G|} \sum_{\sigma \in G} f\circ \theta(\sigma(\boldsymbol{z}))- P(\sigma(\boldsymbol{z}))| < \epsilon 
		\end{align} 
		for $\boldsymbol{z}$ in $K$. Note that $P_G$ is a $G$-invariant polynomial and so, the degree of $P_G$ will be of the form $\boldsymbol{n}=(n,\dots, n)$ for some non-negative integer $n$. A complete argument for this fact has been furnished just before \eqref{Eq: polydegree} in the proof of Theorem \ref{THM_RIF-G}. Now we choose a monomial $$g(\boldsymbol{z})= e^{i\kappa} (z_1 z_2 \dots z_d)^N$$ where the positive integer $N \geq n$ is such that $m$ divides $N$ with 
		\begin{align}\label{Eq:estimations}
			|g(\boldsymbol{z})|<\epsilon \text{ and } |P_G(\boldsymbol{z}) g(\boldsymbol{z}) \overline{P_G\left(\frac{1}{\bar{\boldsymbol{z}}}\right)}|<1
			\text{ on } K.
		\end{align}
		Since $N \geq n$, $Q(\boldsymbol{z})= g(\boldsymbol{z}) \overline{P_G(\frac{1}{\bar{\boldsymbol{z}}})}$ is a monomial times $\tilde{P}_G(\boldsymbol{z})$. We show that $Q$ is a $G$-invariant polynomial. To see this, let $\sigma \in G$ be such that 
		$$\sigma(\boldsymbol{z})= \left( e^{\frac{2\pi i \nu_1 }{m}} z_{\gamma(1)}, \dots, e^{\frac{2\pi i \nu_d }{m}} z_{\gamma(d)}\right)$$
		for some non-negative integers $\nu_1,\dots, \nu_d$ whose sum is divisible by $t$ and $\gamma \in S_d$. Then 
		\begin{align*}
			g(\sigma(\boldsymbol{z}))= e^{i\kappa} e^{\frac{2\pi i}{m} (\nu_1+ \dots + \nu_d)N} (z_1\dots z_d)^N = g(\boldsymbol{z})
		\end{align*}
		as $m$ divides $N$. Now, suppose $P_G(\boldsymbol{z})= \sum_{\boldsymbol{\alpha}} a_{\boldsymbol{\alpha}} \boldsymbol{z}^{\boldsymbol{\alpha}}$. Then
		
		\begin{align*}
			\overline{P_G\left(\frac{1}{\bar{\boldsymbol{z}}}\right)} = \sum_{\boldsymbol{\alpha}} \overline{a_{\boldsymbol{\alpha}}} \frac{1}{z_1^{\alpha_1} \dots z_d^{\alpha_d}}
		\end{align*}
		
		where $\boldsymbol{\alpha}=(\alpha_1, \dots, \alpha_d)$ are multi-indices with non-negative integers $\alpha_j$ for all $j$ and $a_{\boldsymbol{\alpha}}$'s are zero except for finitely many $\boldsymbol{\alpha}$.
		Therefore,
		\begin{align*}
			Q(\sigma(\boldsymbol{z})) &= g(\boldsymbol{z})\sum_{\boldsymbol{\alpha}} \overline{a_{\boldsymbol{\alpha}}} e^{\frac{2\pi i}{m} (\nu_1 \alpha_1+ \dots + \nu_d \alpha_d)} \frac{1}{z_{\gamma(1)}^{\alpha_1} \dots z_{\gamma(d)}^{\alpha_d}} 
		\end{align*}
		Again, $$
		P_G(\boldsymbol{z}) = P_G (\sigma(\boldsymbol{z})) = P_G\left( e^{\frac{2\pi i \nu_1 }{m}} z_{\gamma(1)}, \dots, e^{\frac{2\pi i \nu_d }{m}} z_{\gamma(d)}   \right)
		$$ implies that 
		\begin{align*}
			\overline{P_G\left(\frac{1}{\overline{\boldsymbol{z}}}\right)}&= \overline{P_G \left( e^{\frac{2\pi i \nu_1 }{m}} \frac{1}{\bar{z}_{\gamma(1)}}, \dots, e^{\frac{2\pi i \nu_d }{m}} \frac{1}{\bar{z}_{\gamma(d)}} \right)}\\
			& = \sum_{\boldsymbol{\alpha}} \overline{a_{\boldsymbol{\alpha}}} e^{\frac{2\pi i}{m} (\nu_1 \alpha_1+ \dots + \nu_d \alpha_d)} \frac{1}{z_{\gamma(1)}^{\alpha_1} \dots z_{\gamma(d)}^{\alpha_d}}.
		\end{align*}
		Thus $Q(\sigma(\boldsymbol{z})) = Q(\boldsymbol{z})$. Consider the function $$ \psi_\epsilon (\boldsymbol{z}) = \frac{g(\boldsymbol{z})+ P_G(\boldsymbol{z})}{1+ Q(\boldsymbol{z}) }.$$
		Clearly, $\psi_\epsilon$ is a $G$-invariant rational function. Note that, on $\bT^d$, $|Q|<1$ as $|P_G|<1$ and so, the denominator of $\psi_\epsilon$ is non-vanishing in $\overline{\bD^d}$ which also shows that $\psi_\epsilon \in \cA(\bD^d)$. Also, for $\boldsymbol{z} \in \mathbb{T}^d$ it is clear that $w=P_G(\boldsymbol{z}) \in \bD$, $$\overline{w}=\overline{P_G(\frac{1}{\overline{\boldsymbol{z}}})}\quad  \text{and} \quad |g(\boldsymbol{z})|=1.$$ Thus $|\psi_\epsilon(\boldsymbol{z})|=1$. Moreover, for $\boldsymbol{z} \in K $
		\begin{align*}
			|f\circ \theta(\boldsymbol{z}) - \psi_\epsilon(\boldsymbol{z})| &\leq |f \circ \theta(\boldsymbol{z})-P_G(\boldsymbol{z})| + |P_G(\boldsymbol{z})- \psi_\epsilon(\boldsymbol{z})| \\
			& < \epsilon + |P_G(\boldsymbol{z})- \frac{g(\boldsymbol{z})+ P_G(\boldsymbol{z})}{1+ Q(\boldsymbol{z})}| < 5\epsilon \, (\text{ using } \eqref{Eq:estimate of polynomial} \text{ and } \eqref{Eq:estimations}).
		\end{align*}
		Since the numerator and the denominator of the holomorphic function $\psi_\epsilon$ are $G$-invariant holomorphic polynomials by the Chevalley-Shephard-Todd theorem, there exists a rational function $\Psi_\epsilon$ on $\theta(\bD^d)$ such that $$\psi_\epsilon = \Psi_\epsilon \circ \theta.$$ It is clear that $\Psi_\epsilon$ is a rational inner function in $\cA(\theta(\bD^d))$ and $$|f-\Psi_\epsilon|< 5\epsilon $$ on the compact set $S$. This completes the proof.
	\end{proof} 
	
	A landmark theorem of Fisher \cite{Fisher} says that the uniform closure of the convex hull of finite Blaschke products is the closed norm-unit ball of the disc algebra. See also \cite{Mar} for a version of Fisher's Theorem.  In \cite{Rud_CAG}, Rudin showed that the closed (with respect to sup-norm) convex hull of the inner functions on the polydisc is the norm-unit ball of the polydisc algebra. He proved it in a more general setting, namely, on compact abelian groups (note that $\mathbb{T}^d$ is a compact abelian group). But the \v Silov boundary $\partial \theta(\bD^d)$ of $\theta(\bD^d)$ is not necessarily orientable. For example, when $G=G(1,1,d)$, the \v Silov boundary $\partial \theta(\bD^d)$ of associated quotient domain $\theta(\bD^d)$ which is the symmetrized polydisc, is non-orientable for each even number $d$, see \cite{Morton} and hence $\partial \theta(\bD^d)$ cannot have a topological group structure. Therefore, Rudin's arguments cannot be applied for such quotient domains to get a Fisher type approximation result. Thus, we must approach this problem differently. Motivated by this, we provide an $L^2$-norm approximation on $\partial \theta(\bD^d)$ which is our Theorem \ref{Theorem D}.

	Recall that $H^\infty \left( \mathbb{D}^d \right)^G$ is the collection of all $G$-invariant bounded analytic functions on $\mathbb{D}^d$. Since the radial limits of functions in $H^\infty \left( \mathbb{D}^d \right)$ exist almost everywhere on $\bT^d$, we identify the $H^\infty(\mathbb{D}^d)$ functions with their radial limits and consequently as bounded $\nu$-measurable functions on $\mathbb{T}^d$. Denote by $\mathcal{Q}^G$ the closed convex hull (with respect to the $L^2(\nu)$-norm on $\mathbb{T}^d$) of the $G$-invariant rational inner functions in $\cA(\mathbb{D}^d)$. Recall that, the measure $\mu$ on the \v Silov boundary $\partial\theta(\bD^d)$ is the push-forward by $\theta$ of $\nu$ on $\bT^d$.

	\begin{proof}[\textbf{Proof of Theorem \ref{Theorem D}:}]
		Note that any function $g$ in $ H^\infty(\theta(\bD^d))$ can be identified with the $G$-invariant function $g\circ \theta$ in $H^\infty(\mathbb{D}^d)$ with $$\|g\|_{H^\infty(\theta(\bD^d))} = \|g\circ \theta\|_{H^\infty(\mathbb{D}^d)}.$$ Therefore, it is enough to prove that any function in the closed norm-unit ball of $H^\infty\left(\mathbb{D}^d\right)^G$ can be approximated by convex combination of $G$-invariant rational inner functions in $\cA(\mathbb{D}^d)$ in  $L^2(\mathbb{T}^d, \nu)$. 
		
		If possible, suppose that there exists $f$ in the closed norm-unit ball of  $H^\infty \left( \mathbb{D}^d \right)^G$ such that $f\notin \mathcal{Q}^G$. Then by the Hahn-Banach separation theorem, there exist $g \in L^2(\mathbb{T}^d, \nu)$ and real numbers $r_1$ and $ r_2$ such that
		\begin{align}\label{HBT}
			\operatorname{Re} \int_{\mathbb{T}^d} f(\boldsymbol{\zeta}) g(\boldsymbol{\zeta}) d\nu(\boldsymbol{\zeta}) < r_1 < r_2 < \operatorname{Re} \int_{\mathbb{T}^d} h(\boldsymbol{\zeta}) g(\boldsymbol{\zeta}) d\nu(\boldsymbol{\zeta})
		\end{align}
		for every $h\in \mathcal{Q}^G$. Invoke Theorem $\ref{Theorem C}$ to get  a sequence $\{ u_k\}$ of $G$-invariant rational inner functions in $\cA(\bD^d)$ such that $u_k \longrightarrow f$ uniformly on compact subsets of $\mathbb{D}^d$. Let $\psi$ be a polynomial in $z_1,\dots, z_d$ and $\overline{z_1}, \dots, \overline{z_d}$. Then, the Cauchy integral formula gives the following:
		$$
		\int_{\mathbb{T}^d} u_j(\boldsymbol{\zeta})  \psi(\boldsymbol{\zeta})  d\nu(\boldsymbol{\zeta})  \longrightarrow  \int_{\mathbb{T}^d} f(\boldsymbol{\zeta})  \psi(\boldsymbol{\zeta})  d\nu(\boldsymbol{\zeta}) 
		.$$
		We know that the polynomials in $z_1,\dots, z_d$ and $\overline{z_1}, \dots, \overline{z_d}$ are dense in $L^1\left( \mathbb{T}^d, \nu\right)$. Hence
		$$
		\int_{\mathbb{T}^d} u_j(\boldsymbol{\zeta})  \psi(\boldsymbol{\zeta})  d\nu(\boldsymbol{\zeta})   \longrightarrow  \int_{\mathbb{T}^d} f(\boldsymbol{\zeta})  \psi(\boldsymbol{\zeta})  d\nu(\boldsymbol{\zeta}) 
		$$
		for every $\psi \in L^1\left( \mathbb{T}^d, \nu \right)$. But $g\in L^2(\mathbb{T}^d, \nu) $ and hence $g \in L^1(\mathbb{T}^d,\nu)$. Therefore, we have
		$$
		\int_{\mathbb{T}^d} u_j(\boldsymbol{\zeta})  g(\boldsymbol{\zeta})  d\nu(\boldsymbol{\zeta})   \longrightarrow  \int_{\mathbb{T}^d} f(\boldsymbol{\zeta})  g(\boldsymbol{\zeta})  d\nu(\boldsymbol{\zeta}) .
		$$
		Since $u_j \in \mathcal{Q}^G$, equation \eqref{HBT} implies that
		$$
		\operatorname{Re} \int_{\mathbb{T}^d} f(\boldsymbol{\zeta}) g(\boldsymbol{\zeta})  d\nu(\boldsymbol{\zeta})  \leq r_1 < r_2 \leq \operatorname{Re} \int_{\mathbb{T}^d} f(\boldsymbol{\zeta}) g(\boldsymbol{\zeta})  d\nu(\boldsymbol{\zeta}) 
		$$
		which is a contradiction. Thus, $f\in \mathcal{Q}^G$.  This completes the proof.
	\end{proof}
	
	\section{ Structure of Rational Inner Functions}\label{RIF}
	The goal of this section is to study the rational inner functions on $\theta(\bD^d)$. For a polynomial $\xi \in \mathbb{C}[\boldsymbol{z}]$ with degree $\boldsymbol{n} =(n_1,\dots,n_d)$, the polynomial $\tilde{\xi}$ is defined as $$\tilde{\xi}(\boldsymbol{z})= \boldsymbol{z}^{\boldsymbol{n}} \overline{\xi\left( \frac{1}{\overline{\boldsymbol{z}}} \right)}.$$ 
	We shall start by looking at a description of the \v Silov boundary $\partial \theta(\bD^d)$. Since $\partial \theta(\bD^d)$ is the same as $\theta(\bT^d)$ and $\theta_j$'s are described by \eqref{Eq:theta-j} and \eqref{Eq:theta-d}, we conclude that, if $\boldsymbol{p}=(p_1, \dots, p_d) \in \partial \theta(\bD^d)$ then 
	\begin{align}\label{Eq:Shilov-bdry}
		p_j = \overline{p_{d-j}} p_d^t \text{ for } 1\leq j <d \text{ and } |p_d|=1.
	\end{align}
	
	The following algebraic lemmas will be used in the proof of the main result of this section.
	\begin{lemma}\label{L1}
		Let $\xi $ and $\eta $ be two polynomials in $\bC[\boldsymbol{z}]$ such that $\xi(0,0)\neq 0$ and $\eta(0,0)\neq 0$. Then the following hold:
		\begin{enumerate}
			\item $\widetilde{\widetilde{\xi}} = \xi $ and
			\item $\widetilde{\xi \eta} = \widetilde{\xi}\widetilde{\eta}$.
		\end{enumerate}
		Moreover, if $\xi$ and $\eta$ are two distinct irreducible polynomials, then the following are also true:
		\begin{enumerate}
			\item[(3)] If $\xi$ divides $\widetilde{\xi}$, then $\xi$ is a unimodular scalar multiple of $\widetilde{\xi}$ and
			\item[(4)] If $\xi$ divides $\widetilde{\eta}$ then $\eta$ is a non-zero scalar multiple of $\widetilde{\xi}$.
		\end{enumerate}
	\end{lemma}
	\begin{proof}
		The first two parts are obvious. So, we shall prove $(3)$ and $(4)$.
		Suppose $\xi$ divides $\widetilde{\xi}$. Then there exists $\psi\in\mathbb{C}[\boldsymbol{z}]$ such that $\widetilde{\xi}=\psi\xi$. Thus $\widetilde{\widetilde{\xi}}=\widetilde{\psi}\widetilde{\xi}$. Hence, by part $(1)$, we have $\xi=\widetilde{\psi}\widetilde{\xi}$.
		Since $\xi$ is irreducible, it follows that $\psi$ is a unimodular scalar.\\
		Finally, suppose that $\xi$ divides $\widetilde{\eta}$. Then there exists $\psi\in\mathbb{C}[\boldsymbol{z}]$ such that $\widetilde{\eta}= \psi\xi$. This gives that $$\eta=\widetilde{\widetilde{\eta}}=\widetilde{\psi}\widetilde{\xi}.$$ Since $\eta$ is irreducible, $\psi$ must be a non-zero scalar. This proves $(4)$.
	\end{proof}
	The following lemma about $G$-invariant rational functions on $\mathbb{D}^d$ is crucial for us. 
	\begin{lemma}\label{L2}
		Let $f$ be a $G$-invariant rational function on $\mathbb{D}^d$. Then $f$ can be expressed as a quotient of two $G$-invariant polynomials.
	\end{lemma}
	\begin{proof}
		If $f$ is the identically zero function, then there is nothing to prove. So, we assume $f$ to be non-zero. Since $f$ is a rational function, there exist polynomials $\xi$ and $\eta$ in $\mathbb{C}[\boldsymbol{z}]$ which are co-prime such that
		$f=\frac{\xi}{\eta}.$
		Let $\sigma\in G$. Then $\sigma \cdot f=\frac{\sigma\cdot \xi}{\sigma \cdot \eta }$. Since $f$ is $G$-invariant, we have
		\begin{align}\label{Sym}
			\frac{\xi}{\eta} = f = \sigma \cdot f =\frac{\sigma \cdot \xi }{\sigma \cdot \eta}.
		\end{align}
		Thus, $\xi (\sigma \cdot \eta)=\eta (\sigma \cdot \xi)$. Since $\xi$ and $\eta$ are co-prime, $\xi$ divides $\sigma\cdot \xi $ and $\eta$ divides $\sigma\cdot \eta$. Also, note that the total degrees of $\psi$ and $\sigma \cdot \psi$ are same for any polynomial $\psi\in\mathbb{C}[\boldsymbol{z}]$. Thus,
		$$\sigma\cdot \xi=\lambda_{\sigma}\xi \quad \text{ and } \quad \sigma\cdot \eta=\mu_{\sigma}\eta$$
		for some non-zero scalars $\lambda_{\sigma}$ and $\mu_{\sigma}$. From equation \eqref{Sym}, it follows that $\lambda_{\sigma}=\mu_{\sigma}$.
		The map $$\sigma\mapsto \lambda_{\sigma}$$ defines a group homomorphism from $G$ to the multiplicative group $\mathbb{C}\setminus \{0\}$. But $G$ being a finite group, it is actually a homomorphism from $G$ to $\mathbb{T}$.
		
		If we show that this group homomorphism is the trivial one, then $\sigma\cdot \xi =\xi$ and $\sigma \cdot \eta=\eta$ proving that $\xi$ and $\eta$ are $G$-invariant polynomials. Since the group $G$ is generated by finitely many pseudo-reflections, it is enough to prove that these generators are mapped to $1$ by the aforementioned homomorphism. Let $\sigma$ be such a pseudo-reflection. Then $\operatorname{Ker}(I_d - \sigma)$ is a $d-1$ dimensional subspace of $\bC^d$ and so, there is a linear polynomial $L_\sigma$ in $\bC[\boldsymbol{z}]$ such that $$\operatorname{Ker}(I_d - \sigma) = \{\boldsymbol{z}\in \bC^d: L_\sigma(\boldsymbol{z})=0\}.$$ If possible, suppose that $\lambda_{\sigma}\neq{1}$.
		Then $\sigma\cdot \xi =  \lambda_{\sigma}\xi$ gives
		\begin{align}\label{Eq: invariance}
			\xi(\sigma(\boldsymbol{z})) = \lambda_{\sigma}\xi(\boldsymbol{z}).
		\end{align}
		For $\boldsymbol{z} \in \operatorname{Ker}(I_d - \sigma)$, $\sigma(\boldsymbol{z}) = \boldsymbol{z}$ and so, by \eqref{Eq: invariance}, $\xi(\boldsymbol{z})=0$. Therefore, $L_\sigma$ divides $\xi$. Similarly, $L_\sigma $ also divides $\eta$. But this is not possible as $\xi$ and $\eta$ are co-prime with each other. So, we must have $\lambda_\sigma =1$ and hence the above group homomorphism is trivial.
	\end{proof}
	
	\begin{lemma}\label{L3}
		Let $\xi$ be a polynomial in $\mathbb{C}[\boldsymbol{z}]$. Then $\frac{\tilde{\xi}}{\xi}$ can be written as
		$$ \tau \frac{\tilde{\xi}_{j_1} \tilde{\xi}_{j_2}\dots\tilde{\xi}_{j_l}}{\xi_{j_1}\xi_{j_2}\dots \xi_{j_l}}$$
		for some $\tau\in\mathbb{T}$ and $j_1, j_2,\dots, j_l$ such that $ \tilde{\xi}_{j_1} \tilde{\xi}_{j_2}\dots\tilde{\xi}_{j_l}$ and $\xi_{j_1}\xi_{j_2}\dots \xi_{j_l}$ are co-prime.
	\end{lemma}
	\begin{proof}
		Write $\xi$ as $\xi_1\xi_2\dots \xi_r$ for some irreducible polynomials $\xi_1,\dots, \xi_r\in\mathbb{C}[\boldsymbol{z}]$. Then, by part $(2)$ of Lemma \ref{L1}, we have
		$$\tilde{\xi}=\tilde{\xi}_1 \tilde{\xi}_2\dots \tilde{\xi}_r.$$
		So,
		\begin{align*}
			\frac{\tilde{\xi}}{\xi}=\frac{\tilde{\xi}_1 \tilde{\xi}_2 \dots\tilde{\xi}_r}{\xi_1 \xi_2\dots \xi_r}.
		\end{align*}
		Note that $\tilde{\xi}$ and $\xi$ can have common factors. If $\xi_j$ divides $\tilde{\xi}_j$ for some $j$, then by part $(3)$ of Lemma \ref{L1}, $\frac{\tilde{\xi}_j}{\xi_j} \in \mathbb{T}$.
		Since $\xi_j$'s are irreducible, any such common factor is divisible by $\xi_i$ for some $i$. Without loss of generality, we can assume that  $\xi_1$ divides $\tilde{\xi}_2$. Then by part $(4)$ of Lemma \ref{L1}, $\xi_2 = \beta \tilde{\xi}_1$. Also $\tilde{\xi}_2 = \bar{\beta} \xi_1$. So, in this case, we have
		$$
		\frac{\tilde{\xi}}{\xi} = \frac{\tilde{\xi}_1 ( \bar{\beta} \xi_1 ) \dots \tilde{\xi}_r}{\xi_1 (\beta \tilde{\xi}_1) \dots \xi_r} = \frac{\bar{\beta}}{\beta} \frac{\tilde{\xi}_3 \dots\tilde{\xi}_r}{\xi_3 \dots \xi_r}.
		$$
		After such cancellations, we shall end up with
		$$
		\frac{\tilde{\xi}}{\xi}= \tau \frac{\tilde{\xi}_{j_1} \tilde{\xi}_{j_2}\dots\tilde{\xi}_{j_l}}{\xi_{j_1}\xi_{j_2}\dots \xi_{j_l}}
		$$
		for some $\tau\in\mathbb{T}$ and $j_1, j_2,\dots, j_l$ such that $ \tilde{\xi}_{j_1} \tilde{\xi}_{j_2}\dots\tilde{\xi}_{j_l}$ and $\xi_{j_1}\xi_{j_2}\dots \xi_{j_l}$ are co-prime. This completes the proof.
		
	\end{proof}
	
	\begin{remark}
		It is worth mentioning that in \cite{AMS} the factorization and division results, akin to those in Lemma \ref{L1} and Lemma \ref{L3}, have been carried out in the context of toral and atoral polynomials.  See also \cite{DS} for recent work on this topic from an operator-theoretic perspective. 
	\end{remark}
	
	Now, we are ready to prove the main theorem about the structure of rational inner functions on $\theta(\bD^d)$. This is Theorem \ref{Theorem A}. We restate it for the reader's convenience.
	\begin{thm}
		Given a rational inner function $f$ on $\theta(\bD^d)$, there exist a non-negative integer $k$, $\tau\in\mathbb{T}$ and a polynomial $g$ with no zero in $\theta(\bD^d)$ such that
		\begin{align}\label{THM_RIF-G}
			f(p_1, \dots, p_{d-1}, p_d)=\tau p_d^{k} \ \frac {\overline{g\left(\frac{\overline{p_{d-1}}}{\bar{p}_d^t}, \frac{\overline{p_{d-2}}}{\bar{p}_d^t}, \dots, \frac{\overline{p_1}}{\bar{p}_d^t}, \frac{1}{\bar{p}_d} \right)}}{g\left(p_1,\dots,p_{d-1}, p_d\right)}
		\end{align}
		Conversely, any rational function of the form \eqref{THM_RIF-G} is inner. Moreover, any inner function $f\in \cA(\theta(\bD^d))$ is a rational function of the form \eqref{THM_RIF-G} with the additional property that $g$ has no zeros in $\overline{\theta(\bD^d)}$.
	\end{thm}
	\begin{proof}
		Let $f$ be a rational inner function on $\theta(\bD^d)$. Then $f\circ \theta : \mathbb{D}^d \rightarrow \overline{\mathbb{D}}$ is a $G$-invariant rational inner function on $\mathbb{D}^d$. Invoke Theorem 5.2.5 in \cite{Rudin} to get a polynomial $\xi$ with no zeros in $\mathbb{D}^d$, a $\tau_1\in\mathbb{T}$, and $\boldsymbol{n}=(n_1,n_2,\dots,n_d)$ such that
		\begin{align}\label{RIF-poly}
			f\circ \theta (\boldsymbol{z})= \tau_1 \boldsymbol{z}^{\boldsymbol{n}}\frac{\widetilde{\xi}(\boldsymbol{z})}{\xi(\boldsymbol{z})}.
		\end{align}
		Applying Lemma \ref{L3}, we can find a polynomial $\psi$ such that
		$$
		f \circ \theta = \tau \boldsymbol{z}^{\boldsymbol{n}} \frac{\widetilde{\psi}}{\psi}
		$$
		for some $\tau\in\mathbb{T}$ with $\widetilde{\psi}$ and $\psi$ being co-prime.
		
		Since $\xi$ has no zeros in $\mathbb{D}^d$, $\psi$ cannot have any zero in $\mathbb{D}^d$. Therefore $\boldsymbol{z}^{\boldsymbol{n}} \widetilde{\psi}$ and $\psi$ are co-prime. By applying Lemma \ref{L2}, we conclude that $\boldsymbol{z}^{\boldsymbol{n}} \widetilde{\psi}$ and $\psi$ are both $G$-invariant polynomials. Note that, if $\varphi$ is a $G$-invariant polynomial, then by Chevalley-Shephard-Todd theorem there exists a polynomial $\Xi \in \bC[\boldsymbol{z}]$ such that $\varphi(\boldsymbol{z})= \Xi \circ \theta(\boldsymbol{z})$. If  $\Xi(\boldsymbol{z})= \sum_{\boldsymbol{\alpha}} a_{\boldsymbol{\alpha}}\boldsymbol{z}^{\boldsymbol{\alpha}}$, then 
		\begin{align*}
			\varphi(\boldsymbol{z})&= \sum_{\boldsymbol{\alpha}} a_{\boldsymbol{\alpha}} \prod_{j=1}^d \theta_j(\boldsymbol{z})^{\alpha_j}.
		\end{align*}
		Recall that, $\theta_j(\boldsymbol{z}) = E_j(z_1^m, \dots, z_d^m)$ for $1\leq j \leq d-1$ and $\theta_d(\boldsymbol{z})=(z_1\dots z_d)^{m/t}$. Now, if $\varphi$ is not a constant, then $\Xi$ is non-constant polynomial and suppose the degree of $\Xi$ is $(r_1, \dots, r_d)$. Therefore the degree of $\varphi$ will be of the form:
		\begin{align}\label{Eq: polydegree}
			(n,\dots, n) \text{ with } n= m \sum_{j=1}^{d-1} r_j + \frac{m}{t}r_d.
		\end{align}
		
		Suppose $\boldsymbol{\ell}=(\ell,\dots, \ell)$ is the degree of $\psi$ and it is of the form \eqref{Eq: polydegree} with $\ell= m \sum_{j=1}^{d-1} s_j + \frac{m}{t}s_d$. Then the degree of $\boldsymbol{z}^{\boldsymbol{n}} \widetilde{\psi}$ is $(n_1 + \ell, n_2 + \ell, \dots , n_d +\ell)$. Since  $\boldsymbol{z}^{\boldsymbol{n}} \widetilde{\psi}$ is a $G$-invariant polynomial, for every $k$, $$ n_k +\ell = m \sum_{j=1}^{d-1} r_j + \frac{m}{t}r_d, $$ for some non-negative integers $r_1,\dots, r_d$. Hence, $\boldsymbol{n}$ is of the form \eqref{Eq: polydegree} and we rename it as $\boldsymbol{n}=(\kappa, \dots, \kappa)$ where $\kappa = m \sum_{j=1}^{d-1}(r_j-s_j) + \frac{m}{t}(r_d-s_d)$.
		
		Since $\psi$ is a $G$-invariant polynomial, by Chevalley-Shephard-Todd theorem there exists a polynomial $g$ such that $\psi=g\circ \theta$. Now,
		\begin{align*}
			f\circ\theta(\boldsymbol{z})&= \tau \boldsymbol{z}^{\boldsymbol{n}}\frac {\widetilde{g\circ\theta}(\boldsymbol{z})}{g\circ\theta(\boldsymbol{z})}\\
			&=\tau \boldsymbol{z}^{\boldsymbol{n}} \boldsymbol{z}^{\boldsymbol{\ell}}\frac {\overline{{g\circ\theta}(\frac{1}{\overline{z_1}}, \frac{1}{\overline{z_2}},\dots, \frac{1}{\overline{z_d}})}}{g\circ\theta(z_1, z_2\dots, z_d)}\\
			&=\tau \boldsymbol{z}^{\boldsymbol{n}} \boldsymbol{z}^{\boldsymbol{\ell}} \frac{\overline{g\left(\theta_{1}(\frac{1}{\bar{\boldsymbol{z}}}), \dots,\theta_{d-1}(\frac{1}{\bar{\boldsymbol{z}}}), (\frac{1}{\overline{z_1z_2\dots z_d}})^{m/t}\right)}}{g\left(p_1,\dots,p_{d-1}, p_d\right)}
			\\
			&=\tau \boldsymbol{z}^{\boldsymbol{n}+\boldsymbol{\ell}}\frac {\overline{g\left(\frac{\overline{p_{d-1}}}{\bar{p}_d^t}, \frac{\overline{p_{d-2}}}{\bar{p}_d^t}, \dots, \frac{\overline{p_1}}{\bar{p}_d^t}, \frac{1}{\bar{p}_d} \right)}}{g\left(p_1,\dots,p_{d-1}, p_d\right)}.
		\end{align*}
		Therefore,
		$$f(p_1, \dots, p_{d-1}, p_d)=\tau p_d^{\ell_0} \ \frac {\overline{g\left(\frac{\overline{p_{d-1}}}{\bar{p}_d^t}, \frac{\overline{p_{d-2}}}{\bar{p}_d^t}, \dots, \frac{\overline{p_1}}{\bar{p}_d^t}, \frac{1}{\bar{p}_d} \right)}}{g\left(p_1,\dots,p_{d-1}, p_d\right)}
		$$ where $\ell_0 = m \sum_{j=1}^{d-1}r_j + r_d$.
		
		Conversely, consider a function $f$ on $\theta(\bD^d)$ of the form \eqref{THM_RIF-G}. From \cite[Theorem 5.2.5]{Rudin}, it can be seen that $g$ has no zeros in $\theta(\bD^d) \cup \partial \theta(\bD^d)$ except possibly on a subset of $\partial \theta(\bD^d)$ of $\mu$ measure zero. For $(p_1,\dots, p_d)\in \partial \theta(\bD^d)$, by \eqref{Eq:Shilov-bdry} we have $\overline{p_j}= p_{d-j} \overline{p_d}^t$ for $1\leq j \leq d-1$ and $|p_d|=1$. Therefore, for almost every point $(p_1,\dots,p_d)$ in $\partial \theta(\bD^d)$, $$|f(p_1,\dots, p_d)| =1.$$ Hence $f$ is inner.
		
		Finally, suppose $f\in \cA( \theta(\bD^d))$ is inner. Then  $f \circ \theta \in \cA(\mathbb{D}^d)$ is an inner function. Thus, by Theorem 5.2.5 in \cite{Rudin}, $f\circ \theta$ is of the form \eqref{RIF-poly} such that $\xi$ has no zeros in $\overline{\mathbb{D}^d}$. Again, applying an argument similar to above, we get $f$ to be of the form \eqref{THM_RIF-G} as well as the fact that the polynomial in the denominator has no zeros in $\overline{\theta(\bD^d)}$. This completes the proof.
	\end{proof}
	
	\begin{remark}
		One consequence of the structure of rational inner functions is that there are no non-constant rational inner functions dependent solely on the variables 
		$p_1, p_2, \dots, p_{d-1} $.
		
	\end{remark}
	\section{Comments on operator-valued functions}\label{Operator-Valued}
	This section is dedicated to the case $d=2$. We start by developing the necessary tools to support this transition, specifically the Pick-Nevanlinna interpolation problems.
	
	\subsection{Test functions and the Pick-Nevanlinna interpolation  for $\theta(\bD^2)$:}

	\quad
	
	We begin this subsection by defining the terminology alluded to in the title above. A collection $\Psi$ of $\bC$-valued functions on a set $X$ is called a set of {\em
		test functions} (see \cite{DM}) if the following conditions hold:
	\begin{enumerate}
		\item $\sup_{\psi\in \Psi} |\psi(x)|<1$ for all $x\in X$;
		\item for each finite subset $F$ of $X$, the collection $\{\psi|_{F}: \psi\in \Psi\}$ together with the constant functions generates the algebra of all $\bC$-valued functions on $F$.
	\end{enumerate}
	The second condition is not essential for the development of the theory, see \cite{BH}. The collection $\Psi$ inherits a subspace topology of the space of all bounded functions from $X$ to $\overline{\mathbb D}$ endowed with topology of pointwise convergence. We shall denote the algebra of bounded continuous functions over $\Psi$ with pointwise algebra operation by $C_b(\Psi)$.  Define an injective mapping $E:X\rightarrow C_b(\Psi)$ as $E(x)= e_x$, where $$e_x(\psi)=\psi(x).$$Let $\cH$ be a Hilbert space.  We say that a map $\Lambda: X\times X\rightarrow \mathcal{B}(C_b(\Psi), \mathcal{B}{(\cH}))$ is a completely positive kernel if the following hold:
	\begin{align*}
		\sum_{i,j=1}^n T_j^* \Lambda(x_i, x_j)\left(\overline{f_j}f_i\right)T_i\geq0
	\end{align*}
	for any $n\geq 1$, any $n$ points $x_1, \dots, x_n\in X $, any $n$ operators $T_1, \dots, T_n $ in $\cB(\cH)$ and any $n$ functions $f_1, \dots, f_N$ in $C_b(\Psi)$.
	Let $\mathcal{K}_{\Psi}(\cH)$ be the collection of $\cB(\cH)$-valued positive kernels $k$ on $X$ such that for each $\psi\in\Psi$, the function
	$$(x, y)\mapsto\left(1-\psi(x)\overline{\psi(y)}\right)k(x,y)$$
	is a $\cB(\cH)$-valued positive kernel.  We say that $f: X \rightarrow \cB(\cH)$ is in $H^\infty_{\Psi}(\cH)$ if there is a non-negative constant $C$ such that the function
	\begin{align}\label{HH}
		(x, y)\mapsto\left(C^2 k(x,y) - f(x) k(x,y) f(y)^*\right) 
	\end{align}
	is a $\cB(\cH)$-valued positive kernel for each $k$ in $\mathcal{K}_{\Psi}(\cH)$. If $f$ is in $H^\infty_{\Psi}(\cH)$, then we denote by $C_f$ the smallest $C$ which satisfies \eqref{HH}. The collection of maps $f\in H^\infty_{\Psi}(\cH)$ for which
	$C_f$ is no larger than $1$ is called the {\em $\Psi$-Schur-Agler class} and it is denoted by $\cS\cA_{\Psi}(\cH).$

	It is straightforward to observe that the co-ordinate functions serve as test functions in $\mathbb{D}^d$. The test functions are also known for other domains, see \cite{COMM-LIFT, Tirtha-Hari-JFA, DM, DU}.  The following result identifies a class of holomorphic test functions for $\theta(\mathbb{D}^2)$.
	\begin{lemma}\label{Test}
		The set $\Psi=\{\psi(\beta, \cdot, \cdot): \beta \in \bT\}$ is a collection of test functions for the domain $\theta(\bD^2)$ where 
		$$
		\psi(\beta, p_1, p_2)= \frac{2\beta p_2^t- p_1}{2-\beta p_1} \text{ for } (p_1, p_2) \in \theta(\bD^2).
		$$ 
	\end{lemma}
	\begin{proof}	
		Recall that, the symmetrized bidisc, $\bG_2$ is biholomorphic to $\bD^2/G(1,1,2)$. We first note that the map $\pi: \theta(\bD^2) \rightarrow \bG_2$ given by $\pi(p_1, p_2) = (p_1, p_2^t)$ is a proper holomorphic map. A well known characterization of $\bG_2$ states that, $(q_1, q_2) \in \bG_2$ if and only if $|q_2|<2$ and $|\varphi(\beta, q_1, q_2)|<1$ for every $\beta \in \bT$, where $\varphi(\beta, q_1, q_2)= (2\alpha q_2- q_1)/ (2-\beta q_1)$; see \cite{SYM_GEO}.

		Clearly, for each fixed point $\beta \in\bT$, the map  
		$$\psi(\beta, p_1, p_2)= \varphi(\beta, \cdot, \cdot) \circ \pi(p_1, p_2)= \frac{2\beta p_2^t- p_1}{2-\beta p_1}.
		$$ is holomorphic.  Thus it is easy to verify that $(p_1, p_2) \in \theta(\bD^2)$ if and only if $|p_1|<2$ and $|\psi(\beta, p_1, p_2)|<1$ for each $\beta \in \bT$. Therefore $\Psi=\{\psi(\beta, \cdot, \cdot): \beta \in \bT\}$ is a collection of test functions for the domain $\theta(\bD^2)$.
	\end{proof}
	
	Thus our collection of test functions $\Psi$ is obviously homeomorphic $\bT$. Hence, $C_b(\Psi)$ can be identified with $C(\bT)$. The following result shows that our test functions are rational inner functions.
	
	\begin{lemma}
		For each fixed $\beta\in\bT$, the map	
		$$\psi(\beta, p_1, p_2)=  \frac{2\beta p_2^t- p_1}{2-\beta p_1}$$ 
		is rational inner on $\theta(\bD^2)$.
	\end{lemma}
	\begin{proof}
		Define the polynomial
		$$
		g(p_1, p_2) = 2 - \beta p_1.
		$$
		Note that $g$ has no zeros on $ \theta(\mathbb{D}^2)$. It is straightforward to verify that
		$$
		\overline{g\left(\frac{\overline{p_1}}{\overline{p_2}^t}, \frac{1}{\overline{p_2}}\right)} = \frac{2p_2^t - \overline{\beta}p_1}{p_2^t}.
		$$
		Therefore, our function $\psi(\beta, p_1, p_2)$ can be expressed as
		$$
		\overline{\beta} p_2^t \frac{\overline{g\left(\frac{\overline{p_1}}{\overline{p_2}^t}, \frac{1}{\overline{p_2}}\right)}}{g(p_1, p_2)}.
		$$
		Thus, by Theorem \ref{Theorem A}, the function $\psi(\beta, p_1, p_2)$ is rational inner.
		
	\end{proof}

	As a consequence of the two preceding lemmas, and with the help of Theorem 2.3 from \cite{DM} (see also Theorem 1 from \cite{Tirtha-Ani-Vik}), we can derive the realization formula. However, we will not state it here since it is not needed for our purposes.

	Let $\boldsymbol{z}_1, \boldsymbol{z}_2,\dots,\boldsymbol{z}_n$ be points (initial nodes) in $\Omega\subset\mathbb{C}^2$ and $A_1, A_2,\dots, A_n$ be points (final nodes) in the closed operator-norm unit ball of $\mathbb{M}_{N}(\mathbb{C})$. The Pick-Nevanlinna interpolation problem asks for necessary and sufficient conditions for the existence of a function $g$ which is analytic in $\Omega$ with $\|g(\boldsymbol{z})\|\leq{1}$ for all $\boldsymbol{z}\in\Omega$ and interpolates the data i.e.,
	$$g(\boldsymbol{z}_j)=A_j\quad \text{ for } j=1,2,\dots, n.$$
	When $\Omega=\mathbb{D}$, the problem was solved over a century ago, \cite{Nev, PICK}.  Since then, it has been studied in various domains and different contexts; see \cite{AM-Book,  SYM_Real, Ball-Bolo,  Tirtha-Hari-JFA, Jury-Knese-McC}. The following theorem gives a criteria for solvablity of the Pick-Nevanlinna interpolation problem in $ \theta(\bD^2)$. 
	\begin{lemma}\label{PN}
		Suppose $\boldsymbol{p}_j \in \theta(\bD^2)$ and $A_j \in \mathbb{M}_{N}(\mathbb{C})$ for $1\leq j \leq n$. Then the interpolation problem $\boldsymbol{p}_j \mapsto A_j$ is solvable by a function in $\mathcal{S}\mathcal{A}_{\Psi}(\bC^N)$ if and only if there exists a completely positive kernel 
		$$\Gamma: \{\boldsymbol{p}_1,\dots,\boldsymbol{p}_n \} \times \{\boldsymbol{p}_1, \boldsymbol{p}_2,\dots,\boldsymbol{p}_n \} \rightarrow \mathcal{B}(C_b(\Psi), \cB(\bC^N))$$ such that 
		\begin{align}\label{Eq:interpolation}
			I_N - A_j^* A_i = \Gamma(\boldsymbol{p}_i, \boldsymbol{p}_j) \left(1- \overline{\psi(\cdot,\boldsymbol{p}_j)} \psi(\cdot, \boldsymbol{p}_i) \right) \text{ for } 1\leq i, j \leq n. 
		\end{align}
	\end{lemma}
	The proof follows from Lemma \ref{Test} and Theorem  2.3 of \cite{DM}. \newline

	\begin{definition}
		We say that a rational  map $\Phi=((\Phi_{ij})):\theta(\bD^2) \rightarrow \mathbb{M}_{N}(\mathbb{C})$ is
		
		\item inner if $$\Phi(\boldsymbol{p})^*\Phi(\boldsymbol{p})=I_{N}  \ \text{ for } \mu\text{-a.e. } \boldsymbol{p} \text{ in } \partial\theta(\bD^2).$$
	\end{definition} The following result provides existence of rational inner solution for a solvable Pick-Nevanlinna interpolation data, which we require for the approximation result. For a specific case of this theorem, refer to \cite{DKS2}.
	
	\begin{proposition}\label{Rational-Solution}
		A solvable matrix Pick-Nevanlinna interpolation problem in $\mathcal{S}\mathcal{A}_{\Psi}(\bC^N)$  with initial nodes $\boldsymbol{p}_1,\dots,\boldsymbol{p}_n$  in $\theta(\mathbb{D}^2)$ and the final nodes $A_1, \cdots, A_n$ in the closed operator-norm unit ball of  $\mathbb{M}_{N}(\mathbb{C})$ has a rational inner solution.
	\end{proposition}
	\begin{proof}
		Invoke Lemma \ref{PN} to get a completely positive kernel 
		$$\Gamma: \{\boldsymbol{p}_1,\dots,\boldsymbol{p}_n \} \times \{\boldsymbol{p}_1, \boldsymbol{p}_2,\dots,\boldsymbol{p}_n \} \rightarrow \mathcal{B}(C_b(\Psi), \cB(\bC^N))$$ such that 
		\begin{align*}
			I_N - A_j^* A_i = \Gamma(\boldsymbol{p}_i, \boldsymbol{p}_j) \left(1- \overline{\psi(\cdot,\boldsymbol{p}_j)} \psi(\cdot, \boldsymbol{p}_i) \right) \text{ for } 1\leq i, j \leq n. 
		\end{align*}
		Using Kolmogorov decomposition for $\Gamma$, we have a Hilbert space $\cH$, unital $*$-representation $\Delta : C_b(\Psi) \to \cB(\cH)$ and a function $h:\{\boldsymbol{p}_1, \dots, \boldsymbol{p}_n \} \rightarrow \cB(\bC^N, \cH)$ such that 
		\begin{align}\label{Eq:Kolmogorov}
			\Gamma(\boldsymbol{p}_i, \boldsymbol{p}_j)(\delta)= h(\boldsymbol{p}_j)^* \Delta (\delta)  h(\boldsymbol{p}_i)  \text{ for each } \delta \in C_b(\Psi).
		\end{align}
		Combining \eqref{Eq:interpolation} and \eqref{Eq:Kolmogorov}, we have
		\begin{align}\label{Eq:lurking-isometry}
			I_N + h(\boldsymbol{p}_j)^* \Delta( \overline{\psi(\cdot,\boldsymbol{p}_j)})   \Delta(\psi(\cdot, \boldsymbol{p}_i)) h(\boldsymbol{p}_i)  = A_j^* A_i +  h(\boldsymbol{p}_j)^* h(\boldsymbol{p}_i)
		\end{align}
		Let \begin{align*}
			\cH_1 = \overline{\operatorname{span}} \left\lbrace 
			\begin{pmatrix}
				y \\
				\Delta( \psi(\cdot,\boldsymbol{p}_j)) h(\boldsymbol{p}_j)  y
			\end{pmatrix}: y \in \bC^N \text{ and } 1\leq j \leq n  \right \rbrace
		\end{align*} 
		and \begin{align*}
			\cH_2 = \overline{\operatorname{span}} \left\lbrace 
			\begin{pmatrix}
				A_j y \\
				h(\boldsymbol{p}_j) y
			\end{pmatrix}: y \in \bC^N \text{ and } 1\leq j \leq n  \right \rbrace.
		\end{align*}
		Note that $\cH_1$ and $\cH_2$ are two finite dimensional subspaces of the Hilbert space $\bC^N \oplus \cH$. Again by a lurking-isometry argument, from \eqref{Eq:lurking-isometry} we obtain an isometry $V: \cH_1 \rightarrow \cH_2$ such that 
		\begin{align*}
			V \begin{pmatrix}
				y \\
				\Delta ( \psi(\cdot,\boldsymbol{p}_j)) h(\boldsymbol{p}_j)  y
			\end{pmatrix} = \begin{pmatrix}
				A_j y \\
				h(\boldsymbol{p}_j) y
			\end{pmatrix}
		\end{align*} 
		for every $y\in \bC^N$ and $1\leq j \leq n$.
		
		Thus the isometry can be extended to $\widetilde{V}: \bC^N \oplus \cN \rightarrow \bC^N \oplus \cN $ where $\cN$ is the following finite dimensional subspace of $\cH$:
		\begin{align*}
			\cN =\overline{\operatorname{span}} \left \lbrace \Delta( \psi(\cdot,\boldsymbol{p}_j)) h(\boldsymbol{p}_j) x, h(\boldsymbol{p}_k) y: x, y \in \bC^N \text{ and } j, k=1, \dots, N  \right \rbrace.
		\end{align*}

		We write 
		$$
		\widetilde{V} = \begin{bmatrix}
			A & B\\
			C & D
		\end{bmatrix} : \bC^N \oplus \cN \rightarrow \bC^N \oplus \cN.
		$$
		Thus 
		\begin{align}\label{Eq:realization}
			\begin{bmatrix}
				A & B\\
				C & D
			\end{bmatrix} \begin{pmatrix}
				y \\
				\Delta( \psi(\cdot,\boldsymbol{p}_j)) h(\boldsymbol{p}_j)  y
			\end{pmatrix} = \begin{pmatrix}
				A_j y \\
				h(\boldsymbol{p}_j) y
			\end{pmatrix} \text{ for } y\in \bC^N \text{ and } 1\leq j \leq n.
		\end{align} 
		Consider the function 
		\begin{align}
			f(\boldsymbol{p}) = A + B \Delta(\psi(\cdot, \boldsymbol{p})) \left(I- D \Delta (\psi(\cdot, \boldsymbol{p})) \right)^{-1} C. 
		\end{align}
		Since for $\boldsymbol{p} \in \theta(\bD^2)$, $|\psi(\cdot, \boldsymbol{p})|<1$ and hence $\|D \Delta (\psi(\cdot, \boldsymbol{p}))\|<1$.  It is now standard to verify that $f$ is a bounded holomorphic function from $\theta(\mathbb{D}^2)$ into $\mathbb{M}_N(\mathbb{C})$ with a sup-norm no greater than 1, such that, with the help of \eqref{Eq:realization}, it interpolates the given data, i.e., $f(\boldsymbol{p_j}) = A_j$ for each $1 \leq j \leq n$.

		Finally, we note that the function $\psi(\beta, \boldsymbol{p})$ has singularity when $ p_1 =2 \bar{\beta}$ where $\boldsymbol{p}=(p_1, p_2)$ and it can be easily seen that the set $\mathcal{G}= \{\boldsymbol{q}=(q_1, q_2) \in \partial \theta(\bD^2): |q_1|=2 \}$ is a $\mu$-measure zero subset of $\partial \theta(\bD^2)$. Also, for each $\boldsymbol{p} \in \partial \theta(\bD^2) \setminus \mathcal{G}$ and $\beta \in \bT$, 
		$$
		|\psi(\beta, \boldsymbol{p})| = 1. 
		$$
		So, $\Delta(\cdot, \boldsymbol{p})$ is unitary operator for $\boldsymbol{p} \in \partial \theta(\bD^2) \setminus \mathcal{G}$. Therefore,
		$$
		f(\boldsymbol{p})^* f(\boldsymbol{p})=I_N \text{ for } \boldsymbol{p} \in \partial \theta(\bD^2) \setminus \mathcal{G}
		$$
		and hence $f$ is a rational  inner solution of the Pick-Nevanlinna interpolation in the statement of the proposition.
	\end{proof}

	\subsection{Approximation for operator-valued case}
	Recently, it has been established that if the Carath\'eodory approximation theorem holds for the matrix-valued case, then the result can be extended to the operator-valued case using approximation arguments, see \cite[Theorem 3.6]{BBK}. Therefore, we will focus on discussing the Carath\' eodory approximation theorem for the matrix-valued case. It is folklore that Carath\'eodory's approximation theorem for $\bD$ can be proven via the Pick-Nevanlinna interpolation problem, which extends naturally to the matrix-valued case, see for example \cite{B_J_K}. Having Proposition \ref{Rational-Solution} in hand,  we have the following result.
	\begin{thm}
		Any holomorphic function $f$ in $\mathcal{S}\mathcal{A}_{\Psi}(\bC^N)$, can be approximated (uniformly on compact subsets) by matrix-valued rational inner functions. 
	\end{thm}
	\begin{proof}
		We sketch the proof. Choose a countable dense subset $\{ \boldsymbol{\lambda}_1, \boldsymbol{\lambda}_2, \dots \}$ of $\theta(\mathbb{D}^2)$ and set up the interpolation problem $\{ \boldsymbol{\lambda}_j \to f(\boldsymbol{\lambda}_j) \}_{j=1}^n$, for each $n \geq 1$. Since the data is solvable by $f$, according to Proposition \ref{Rational-Solution}, there exists a rational inner  function $f_n$ on $\theta(\mathbb{D}^2)$ such that
		$$f_n(\boldsymbol{\lambda}_j)=f(\boldsymbol{\lambda}_j) \quad \text{ for all } j=1,2,...,n.$$
		Note that each $f_n$ is bounded by $1$. Thus, the family $$\cF:=\{f_n: n\in\mathbb{N}\}$$ is uniformly bounded. By Montel's Theorem, there exists a subsequence of $\cF$ that converges uniformly on each compact subset of $\theta(\mathbb{D}^2).$ An application of Arzela-Ascoli theorem completes the proof.
	\end{proof}
	\begin{remark}
		It is currently unknown whether 
		$\mathcal{S}\mathcal{A}_{\Psi}(\mathbb{C}^N)$ 
		is the same as the norm unit ball of $N \times N$ matrix-valued holomorphic functions on $\theta(\mathbb{D}^2)$. Establishing this requires proving the existence of rational dilation. However, since the existence of rational dilation is not known, we plan to investigate this further in future work.
	\end{remark}

	\noindent \textbf{Funding:} The first author's work is supported by the Prime Minister's Research Fellowship, Government of India (PMRF-21-1274.03), while the second author is currently supported by a PIMS Postdoctoral Fellowship.
	
	A particular case of this work was initiated during the second author's tenure as a research associate in the Department of Mathematics at the Indian Institute of Science, Bangalore. During this period, his research was supported by the INSPIRE grant (Ref: DST/INSPIRE/04/2019/000769), awarded to Dr. Srijan Sarkar by the Department of Science \& Technology (DST), Government of India

\end{document}